\documentclass[12pt,letterpaper]{article}

\pdfoutput=1

\usepackage{graphics,epsfig,graphicx,color}
\graphicspath{{../Figures/}{Figures/}}
\usepackage[caption = false]{subfig}
\usepackage{float}
\usepackage{cite}
\usepackage{capt-of}

\usepackage{amssymb,latexsym}

\usepackage{setspace}

\usepackage{amsmath}

\usepackage{mathrsfs,amsfonts}

\usepackage{amsthm,amsxtra}

\usepackage[final]{showkeys}

\usepackage{stmaryrd}

\usepackage[ruled,vlined,noresetcount]{algorithm2e}
\usepackage{algorithmicx}
\usepackage{algpseudocode}

\usepackage{etaremune}
\usepackage{enumitem} 

\newif\ifPDF
\ifx\pdfoutput\undefined
\PDFfalse
\else
\ifnum\pdfoutput > 0
\PDFtrue
\else
\PDFfalse
\fi
\fi

\ifPDF
\usepackage{pdftricks}
\begin{psinputs}
    \usepackage{pstricks}
    \usepackage{pstcol}
    \usepackage{pst-plot}
    \usepackage{pst-tree}
    \usepackage{pst-eps}
    \usepackage{multido}
    \usepackage{pst-node}
    \usepackage{pst-eps}
\end{psinputs}
\else
\usepackage{pstricks}
\fi

\ifPDF
\usepackage[debug,pdftex,colorlinks=true, 
linkcolor=blue, bookmarksopen=false,
plainpages=false,pdfpagelabels]{hyperref}
\else
\usepackage[dvips]{hyperref}
\fi

\usepackage{fancyhdr}

\usepackage{comment}

\usepackage[toc,page]{appendix}

\usepackage{cancel}

\usepackage{multirow}

\newtheorem{theorem}{Theorem}[section]
\newtheorem{lemma}[theorem]{Lemma}
\newtheorem{definition}[theorem]{Definition}

\newtheorem{remark}[theorem]{Remark}

\newcommand{\sgn}{\operatorname{sgn}}

\newcommand{\sfc}{\mathsf c}

 \newcommand{\bbN}{\mathbb N}
 
\newcommand{\bbR}{\mathbb R} \newcommand{\bbS}{\mathbb S}

\newcommand{\btheta}{{\boldsymbol\theta}}

\newcommand{\ba}{\mathbf a} \newcommand{\bb}{\mathbf b}

 \newcommand{\bn}{\mathbf n}
 
 \newcommand{\br}{\mathbf r}
 
 \newcommand{\bv}{\mathbf v} 
 \newcommand{\bx}{\mathbf x} 
\newcommand{\by}{\mathbf y} \newcommand{\bz}{\mathbf z}

\newcommand{\cA}{\mathcal A} \newcommand{\cB}{\mathcal B}

\newcommand{\cG}{\mathcal G} 
\newcommand{\cI}{\mathcal I} 
\newcommand{\cK}{\mathcal K} \newcommand{\cL}{\mathcal L}
 
\newcommand{\cO}{\mathcal O}

\newcommand{\aver}[1]{\langle {#1} \rangle}

\newenvironment{keywords}
{\noindent{\bf Key words.}\small}{\par\vspace{1ex}}
\newenvironment{AMS}
{\noindent{\bf AMS subject classifications 2010.}\small}{\par}

\title{A fast algorithm for time-dependent radiative transport equation based on integral formulation}

\author{
    Hongkai Zhao\thanks{
        Department of Mathematics, University of California, Irvine, CA 92697; zhao@uci.edu}
    \and
    Yimin Zhong\thanks{
        Department of Mathematics, University of California, Irvine, CA 92697; yiminz@uci.edu}
}

\begin{document}
\maketitle



\begin{abstract}
    In this work, we introduce a fast numerical algorithm to solve the time-dependent radiative transport equation (RTE). Our method uses the integral formulation of RTE and applies the treecode algorithm to reduce the computational complexity from $\cO(M^{2+1/d})$ to $\cO(M^{1+1/d}\log M)$, where $M$ is the number of points in the physical domain. The error analysis is presented and numerical experiments are performed to validate our algorithm. 
\end{abstract}


\begin{keywords}
    radiative transport equation, volume integral equation, treecode algorithm
\end{keywords}


\begin{AMS}
    45K05, 65N22, 65N99, 65R20, 65Y10
\end{AMS}

\section{Introduction}
The radiative transport model plays an important role in quantitative modeling and analysis of particle transport processes in many physical and biological applications such as astrophysics~\cite{cecchi1999radiative,henyey1941diffuse}, nuclear engineering~\cite{larsen1988neutronics,mokhtar1997mathematical}, biomedical optics~\cite{arridge2009optical,ren2015inverse,bal2016ultrasound, zhao2019instability,li2019inverse}, radiation therapy~\cite{jia2012gpu,tillikainen20083d}. In this paper, we consider the numerical solution to the time-dependent radiative transport equation (RTE) with isotropic scattering kernel:
\begin{equation}\label{EQ: TIME RTE}
    \begin{aligned}
        u_t(t, \bx, \bv) + \left[\bv\cdot\nabla +\sigma_t(\bx) \right]u(t, \bx, \bv) &= \sigma_s(\bx) \aver{ u}(t, \bx) + f(t, \bx)&\text{in }& (0,T]\times \Omega\times\bbS^{d-1}  \,, \\
                u(t, \bx, \bv) & = 0 &\text{on }&   \{0\}\times \Omega\times\bbS^{d-1}\,,\\
        u(t, \bx, \bv) &= 0 &\text{on }& (0,T]\times \Gamma_{-}\,.
    \end{aligned}
\end{equation}
where the space $\Omega\subset \bbR^d$ is a convex domain with smooth boundary $\partial\Omega$, $\bbS^{d-1}$ denotes the unit sphere in $\bbR^d$. $\Gamma_{-} := \{(\bx, \bv)\in \partial \Omega\times \bbS^{d-1}\mid \bv\cdot \bn_{\bx} < 0\}$ ($\bn_{\bx}$ being the
unit outward normal at $\bx\in\partial\Omega$) is the incoming boundary set. $\sigma_t(\bx)$ and $\sigma_s(\bx)$ are the total absorption and scattering coefficients, respectively. Physically speaking, the coefficient $\sigma_s(\bx)$ represents the strength of the scattering of the underlying medium at $\bx\in\Omega$ and $\sigma_a(\bx):=\sigma_t(\bx) - \sigma_s(\bx)$ represents the strength of absorption of the medium. $f(t,\bx)$ is a time-dependent isotropic  source function (which is not dependent on $\bv$). The quantity $\aver{u}(t,\bx)$ is defined by
\begin{equation}
    \aver{u}(t,x) := \int_{\bbS^{d-1}} u(t, \bx, \bv') d\bv'\,,
\end{equation}
where $d\bv'$ is the \emph{normalized} surface measure on $\bbS^{d-1}$. For the sake of simplicity, we have assumed there is no incoming source on the boundary, and the solution $u(t,\bx, \bv)$ is zero at $t = 0$.


The analytic solutions for the time-dependent RTE~\eqref{EQ: TIME RTE} have only been found in special setup, such as for homogeneous infinite or semi-infinite geometries~\cite{liemert2011analytical,paasschens1997solution, elaloufi2002time}, and layered media~\cite{liemert2017analytical}. Numerical methods for solving~\eqref{EQ: TIME RTE} has been extensively
explored, see~\cite{guo2002monte,tan2001integral,larsen1988grey,graziani2006computational,lewis1984computational} and references therein for an overview. These numerical algorithms are mainly based on stochastic Monte Carlo~\cite{guo2002monte,gentile2001implicit,mcclarren2009modified}, discrete ordinate~\cite{guo2002three,chai2004three,gaggioli2019light, ling2018conservative,kim2002chebyshev}, or $P_N$ formulation~\cite{powell2019pseudospectral, egger2016class}. The most challenging issue for solving the RTE numerically is due to the high dimensionality of the phase space that includes both physical and angular dimensions. 
Regarding time-independent problems, 
one of the popular ways is based on the integral formulation to remove the angular variable by computing the angular moments~\cite{ren2019fast, fan2019fast,ren2019separability}. 
For isotropic scattering media, the fast algorithms based on fast multipole method~\cite{ren2019fast} and low rank matrix factorization~\cite{fan2019fast} were developed. For anisotropic scattering media, a truncated coupled system of integral equations for the angular moments of the transport solution were studied in~\cite{ren2019separability}. Particularly, for those highly separable scattering phase functions such as Rayleigh or linearly anisotropic cases, the integral formulation could solve the RTE very effectively by exploiting the low rank structure of integral kernels~\cite{ren2019separability}. Regarding time-dependent problems as~\eqref{EQ: TIME RTE}, the integral formulation for infinite homogeneous medium has been carried out in~\cite{tan2001integral,wu2000integral}, however the related fast algorithms have not been addressed yet.

In our work, we will pursue the integral formulation for angular averaged solution for time-dependent RTE and develop a fast solver based on the treecode algorithm for the resulting integral equation in space and time, which is more complicated due to the manifold structure, a conical surface, for the domain of dependence. We will briefly derive the integral formulation in Section \ref{sec:formulation} and provide a few mathematical preliminaries in Section \ref{sec:preliminary}. Then we present our fast algorithm including discretization, error analysis, and implementation details in Section \ref{sec:algorithm}. We provide numerical experiments in Section \ref{sec:test} and concluding remarks in Section \ref{sec:conclusion}.

\section{Integral formulation}
\label{sec:formulation}
In this section, we first briefly introduce the integral formulation for the time-dependent RTE~\eqref{EQ: TIME RTE}. Let 
\begin{equation}
\bz:= (t, \bx)\in \bbR^{d+1},\quad \btheta:= (1, \bv)\in \{1 \}\times \bbS^{d-1}\,.
\end{equation}
We slightly abuse the notations without causing any confusion that $u(\bz, \btheta) = u(t,\bx, \bv)$, $f(\bz) = f(t,\bx)$, $\sigma_s(\bz) = \sigma_s(\bx)$,  $\sigma_t(\bz) = \sigma_t(\bx)$. Let $R(\bz) =\sigma_s(\bz)\aver{u}(\bz) + f(\bz)$, then the RTE~\eqref{EQ: TIME RTE} can be formulated as a usual linear transport equation:
\begin{equation}\label{EQ: LINEAR TRANSPORT}
{\btheta}\cdot \nabla_{\bz} u(\bz, {\btheta}) + \sigma_t(\bz) u(\bz, {\btheta}) = R(\bz)
\end{equation}
with the initial and boundary conditions in~\eqref{EQ: TIME RTE}. Under the convention that $u(t,\bx, \bv) = 0$ and $f(t,\bx) = 0$ for $t < 0$, we can solve the linear transport equation~\eqref{EQ: LINEAR TRANSPORT} by
\begin{equation}\label{EQ: LINEAR TRANSPORT SOLUTION}
u(\bz, \btheta) = \int_0^{\tau_{-}(\bx, \bv)} \exp\left(-\int_0^r \sigma_t(\bx - s\bv) ds\right) R(\bz - r{\btheta}) dr\,,
\end{equation}
where $\tau_{-}(\bx, \bv)$ is the distance from the location $\bx$ to reach the boundary $\partial\Omega$ along the direction $-\bv$, which is:
\begin{equation}
\tau_{-}(\bx, \bv) := \sup\{r \mid \bx - r' \bv\in\Omega\;\text{  for  }\;0\le r' < r\}\,.
\end{equation}
Integrate both sides of~\eqref{EQ: LINEAR TRANSPORT SOLUTION} over the angular variable $\bv\in \bbS^{d-1}$, we will obtain
\begin{equation}\label{EQ: INT EQ LINE}
\begin{aligned}
\aver{u}(t,\bx)  =\int_0^{\tau_{-}(\bx, \bv)} \exp\left(-\int_0^r \sigma_t(\bx - s\bv) ds\right) R\left(t-r, \bx - r\bv\right)dr d\bv\,.
\end{aligned}
\end{equation} 
To further simplify the equation~\eqref{EQ: INT EQ LINE},  let $\by = \bx - r\bv$, which means $r = |\bx- \by|$ and $\bv = (\bx-\by)/|\bx - \by|$, and define the function $E(\bx, \by)$ as:
\begin{equation}
\begin{aligned}
E(\bx, \by) = \exp\left(-\int_0^{|\bx - \by|}\sigma_t\left(\bx - s\frac{\bx - \by}{|\bx - \by|}\right) ds\right)\,,
\end{aligned}
\end{equation}
which is the total attenuation due to absorption along the line segment between $\bx$ and $\by$ in $\Omega$. Use the transformation between Cartesian and polar coordinates,
\begin{equation}
d\by = \nu_{d-1}r^{d-1} dr d\bv\,,
\end{equation}
with $\nu_{d-1}$ as the surface area of the unit sphere $\bbS^{d-1}$, the equation~\eqref{EQ: INT EQ LINE} can be rewritten as 
\begin{equation}\label{EQ: INT EQ}
\aver{u}(t,\bx) = \frac{1}{\nu_{d-1}}\int_{\Omega} \frac{E(\bx, \by)}{|\bx - \by|^{d-1}} \left( \sigma_s\aver{u}(t - {|\bx - \by|}, \by)  + f(t-{|\bx - \by|}, \by)  \right) d\by\,.
\end{equation}
Geometrically, the equation~\eqref{EQ: INT EQ} describes that the solution $\aver{u}(t,\bx)$ is an integral over the conical surface formed by the characteristic lines in the cylinder $[0,T]\times\Omega$, see Fig~\ref{FIG: CONICAL}.
\begin{figure}[!htb]
    \centering
    \includegraphics[scale=0.15]{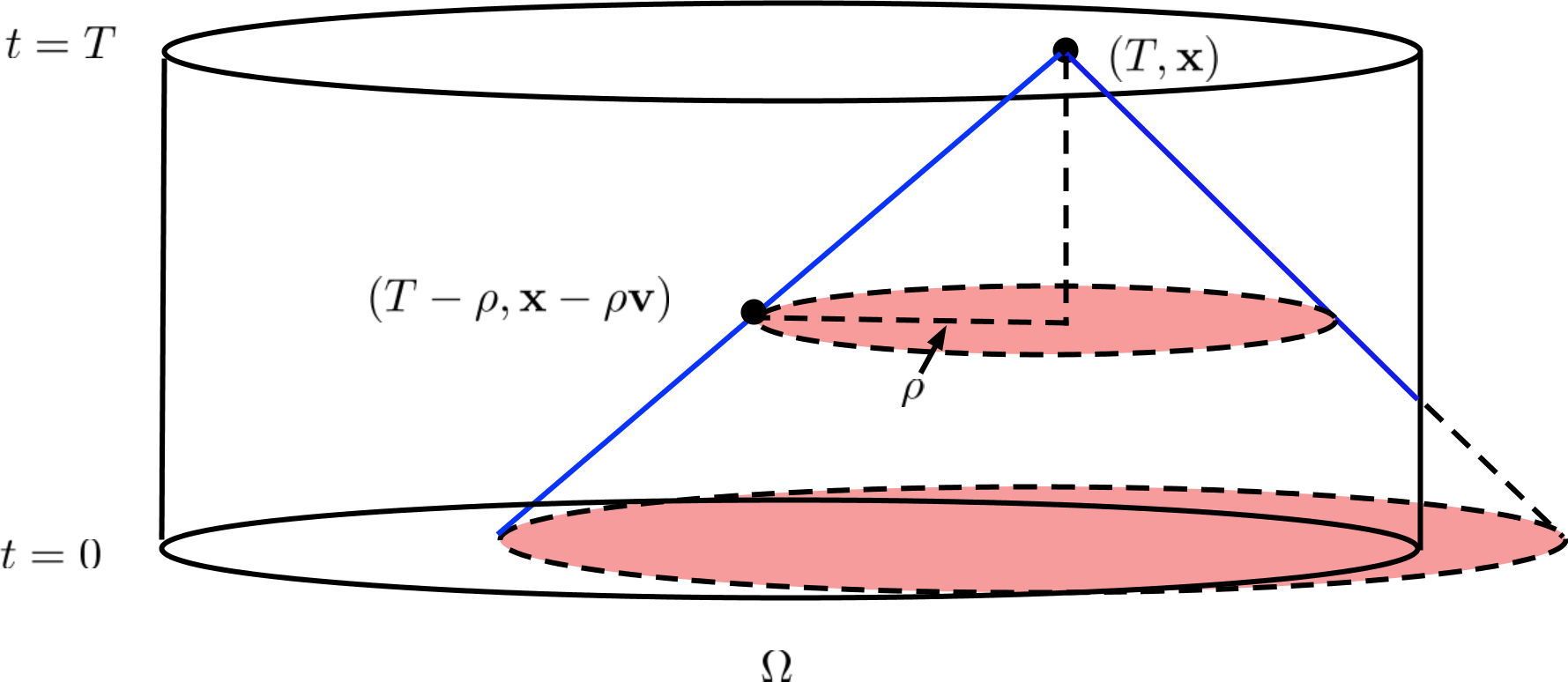}
    \caption{The characteristic line (blue) in the phase space}
    \label{FIG: CONICAL}
\end{figure}
Our main numerical algorithm will be based on the integral formulation~\eqref{EQ: INT EQ} for the angular averaged quantity $\aver{u}$. In the following context, we assume the coefficients  $\sigma_t(\bx), \sigma_s(\bx)$ and the source function $f(t,\bx)$ satisfy the following conditions.
\begin{enumerate}
    \item [$\cA$.] The coefficients $\sigma_s(\bx),\sigma_t(\bx)\in C^{1,{\alpha}}(\overline{\Omega})$ for some $\alpha\in[0,1]$, and there exist constants  $k_0$, $\underline{\sfc}$ and $\overline{\sfc}$ such that 
    \begin{equation}\nonumber
0\le \sup_{\Omega} \frac{\sigma_s}{\sigma_t} \le  k_0 <1,\quad 0 < \underline{\sfc}\le \sigma_s \le \overline{\sfc}\,.
    \end{equation}
    \item [$\cB$.] The source function $f\in C([0,T]\times\Omega)$ and $f(0,\cdot) \equiv 0$. There exists a constant $C_f$ that  satisfies 
    \begin{equation}
    |f(t, \bx) - f(s, \by)| \le C_f \left(\omega(|s-t|) + \omega(|\bx - \by|) \right),\forall (t,\bx), (s,\by)\in [0,T]\times\Omega\,,
    \end{equation}where $\omega(s):= s(1+|\log s|)$.
\end{enumerate}

\section{Mathematical Preliminaries} 
\label{sec:preliminary}
In this section, we provide a few basic but useful properties for the solution to the integral equation~\eqref{EQ: INT EQ}.  

\begin{lemma}\label{LEM: WELL}
    Suppose the assumption $\cA$ is satisfied and $f(t,\bx)\in L^{\infty}([0,T]\times\Omega)$, then there exists a unique solution $\aver{u}\in L^{\infty}([0,T]\times\Omega)$  to~\eqref{EQ: INT EQ}. Moreover, there exists a constant $C_0$ such that
    \begin{equation}
    \|\aver{u}\|_{L^{\infty}([0,T]\times\Omega)} \le C_0 \|f\|_{L^{\infty}([0,T]\times\Omega)}\,.
    \end{equation}
\end{lemma}
\begin{proof}
    Define the operator $\cK: L^{\infty}([0,T]\times\Omega)\mapsto L^{\infty}([0,T]\times\Omega)$
    \begin{equation}
    \cK U(t,\bx) = \frac{1}{\nu_{d-1}}\int_{\Omega} \frac{E(\bx, \by)}{|\bx - \by|^{d-1}}\sigma_s(\by) U(t-|\bx - \by|, \by) d\by\,,
    \end{equation}
    then $\|\cK\|_{op} \le k_0 < 1$, where $\|\cdot\|_{op}$ means the operator norm, which implies $\cK$ is a contraction. The solution $\aver{u}$ can be solved through
    \begin{equation}
    \aver{u} = (\cI - \cK)^{-1}\cK \left(\frac{f}{\sigma_s}\right)\,.
    \end{equation}
    Hence the conclusion follows by the Banach fixed point theorem that
    \begin{equation}
    \|\aver{u}\|_{L^{\infty}([0,T]\times\Omega)} \le \frac{k_0}{\underline{\sfc}(1 - k_0)}\|f\|_{L^{\infty}([0,T]\times\Omega)}\,.
    \end{equation}
\end{proof}
\begin{lemma}\label{LEM:TIME}
    Suppose the assumptions $\cA$ and $\cB$ are satisfied, then
    \begin{equation}
    |\aver{u}(t, \bx) - \aver{u}(s, \bx)|\le C_0 C_f \omega(|s-t|)\,.
    \end{equation}
\end{lemma}
\begin{proof}
     Define 
    \begin{equation}
    w(\tau,\bx) = \aver{u}(\tau,\bx) - \aver{u}(\tau - (t-s), \bx)\,,
    \end{equation}
    then $w(t, \bx)=\aver{u}(t, \bx) - \aver{u}(s, \bx)$ and $w$ satisfies
    \begin{equation}
    \begin{aligned}
    w(\tau,\bx) &=\frac{1}{\nu_{d-1}}\int_{\Omega} \frac{E(\bx, \by)\sigma_s(\by)}{|\bx - \by|^{d-1}}w(\tau-|\bx - \by|, \by) d\by \\
    &+ \frac{1}{\nu_{d-1}}\int_{\Omega} \frac{E(\bx, \by)}{|\bx - \by|^{d-1}} \left[f(\tau-|\bx - \by|, \by) - f(\tau - (t-s)-|\bx-\by|,\by)\right] d\by\,.
    \end{aligned}
    \end{equation}
Since $f(0,\cdot) = 0$, then with the convention that $f\equiv 0$ for $t < 0$ and monotonicity of $\omega(s)= s(1+|\log s|)$, we obtain 
\begin{equation}
|f(\tau-|\bx - \by|, \by) - f(\tau - (t-s)-|\bx-\by|,\by) |\le C_f\omega( |s-t|)\,,
\end{equation}
By Lemma~\ref{LEM: WELL}, we conclude that $\|w\|_{L^{\infty}([0,T]\times\Omega)} \le C_0 C_f \omega(|s-t|)$.
\end{proof}
\begin{lemma}\label{LEM: SPACE}
    Suppose the assumptions $\cA$ and $\cB$ are satisfied, then there exists a constant $C$ such that
    \begin{equation}
    |\aver{u}(t, \bx) - \aver{u}(t, \bz)|\le C \omega(|\bx - \bz|)\,.
    \end{equation}
\end{lemma}
\begin{proof}
Denote $R = \sigma_s\aver{u} + f$, then we  can decompose
    \begin{equation}
    \begin{aligned}
    \aver{u}(t,\bx) - \aver{u}(t,\bz) = \cI_1 + \cI_2 +\cI_3\,,
    \end{aligned}
    \end{equation}
    where $\cI_i$, $i=1,2,3$ are:
    \begin{equation}
    \begin{aligned}
    \cI_1 &= \frac{1}{\nu_{d-1}}  \int_{\Omega} \left[ \frac{E(\bx, \by)}{|\bx- \by|^{d-1}} - \frac{E(\bz,\by)}{|\bz - \by|^{d-1}} \right]R(t-|\bx-\by|,\by) d\by\,, \\
    \cI_2 &= \frac{1}{\nu_{d-1}}  \int_{\Omega} \frac{E(\bz, \by)\sigma_s(\by)}{|\bz- \by|^{d-1}}\left[\aver{u}(t-|\bx-\by|,\by) - \aver{u}(t-|\bz-\by|,\by) \right]d\by\,,\\
    \cI_3 &= \frac{1}{\nu_{d-1}}  \int_{\Omega} \frac{E(\bz, \by)}{|\bz- \by|^{d-1}}\left[f(t-|\bx-\by|,\by) - f(t-|\bz-\by|,\by) \right]d\by\,.
    \end{aligned}
    \end{equation}
    By Lemma~\ref{LEM:TIME} and the assumption $\cB$, $|\cI_2|\le C_0 C_f\omega(|\bx - \bz|)$ and $|\cI_3|\le C \omega(|\bx - \bz|)$. For $\cI_1$, by the Lemma 2.3 from~\cite{vainikko2006multidimensional} that $
    |\cI_1|\le C\omega(|\bx - \bz|)$, therefore $|\aver{u}(t,\bx) - \aver{u}(t,\bz)| \le C \omega(|\bx -\bz|)$.
\end{proof}

\section{Numerical algorithm}
\label{sec:algorithm}
In the next, we will develop an efficient numerical algorithm to solve the integral equation~\eqref{EQ: INT EQ}. When $\aver{u}(t, \bx)$ is known, the solution $u(t, \bx, \bv)$ can be easily computed by~\eqref{EQ: LINEAR TRANSPORT SOLUTION} using a fast sweeping method~\cite{gao2009fast}. The main advantage of this algorithm is that it does not require an explicit discretization for the angular variable $\bv$. It is clear that the computational cost for~\eqref{EQ: INT EQ} will be only depending on the time and spatial variables. In many practical applications such as radiation hydrodynamics~\cite{densmore2012hybrid}, astrophysical plasmas, the main quantities of interests are not the local solutions $u(t,\bx, \bv)$. In these cases, we do not even need to perform the computation for~\eqref{EQ: LINEAR TRANSPORT SOLUTION} and the computational complexity will be completely independent of the angular variables. 

\subsection{Discretization of time}
\label{SEC:TIME}
\begin{definition}
    Let $S_i = [t_{i}, t_{i+1}]$, $i=0,\dots, N$ as an equispaced subdivision of $[0,T]$ that
    \begin{equation}
    0 = t_{0} < \dots < t_{N} =T ,\quad  t_{i+1} - t_{i}\equiv h\,,
    \end{equation}
    then $t_i = i h$. Denote the piecewise linear continuous function on $[0,T]$ as 
    \begin{equation}
    V_{h} = \{\psi\in C([0,T]): \psi|_{S_i}\in P_1(S_i),\;\forall 1\le i\le N \}\,,
    \end{equation}
    where $P_1(S_i)$ is the set of linear polynomials on $S_i$.The space $V_h$ can be spanned by the nodal basis $\{\phi_l\}_{l=0}^{N+1}\subset V_h$ where $\phi_l(t_j) = \delta_{lj}$, where $\delta_{lj}$ is the Kronecker delta.
\end{definition}
We seek for the time domain piecewise linear solution $w_{h}(t,\bx)$ to the equation~\eqref{EQ: INT EQ}. Let $\aver{u}_h$ and $f_h$ be the corresponding approximations for $\aver{u}$ and $f$ in the form as
\begin{equation}\label{EQ:APPROX TIME}
\aver{u}_h(t,\bx) = \sum_{l= 0}^{N} w_{l}(\bx) \phi_{l}(t),\quad f_h(t,\bx) = \sum_{l= 0}^{N} c_{l}(\bx) \phi_{l}(t)\,, 
\end{equation}
where $w_{l}(\bx)=\aver{u}_h(t_l,\bx), c_l(\bx) = f(t_l, \bx)$ by the definition of $\phi_l$. At the time step $t=t_l$, $w_l(\bx)$ satisfies the following integral equation instead,
\begin{equation}\label{EQ: WL0}
w_l(\bx) = \frac{1}{\nu_{d-1}}\sum_{k=0}^N\int_{\Omega} \frac{E(\bx, \by)}{|\bx - \by|^{d-1}}  (\sigma_s(\by) w_k(\by)\phi_k(t_l-{|\bx - \by|}) + c_k(\by)\phi_k(t_l-{|\bx-\by|})) d\by\,.
\end{equation}
Define the standard hat function $V(t)$ supported on $[-1,1]$ that
\begin{equation}
V(t) = 
\begin{cases} 
1 - t,\quad &\text{if } t\in [0,1]\,,\\
1 + t,\quad &\text{if } t\in [-1,0)\,,\\
0,\quad&\text{otherwise}\,,
\end{cases}
\end{equation} 
then $\phi_k(t_l - |\bx - \by|)$ can be represented by
\begin{equation}
\phi_k(t_l - |\bx - \by|) = V\left(\frac{|\bx-\by|}{h} + k-l \right)\,.
\end{equation}
The integral equation~\eqref{EQ: WL0} is then rewritten as
\begin{equation}\label{EQ: WL}
\begin{aligned}
w_l(\bx) &= \frac{1}{\nu_{d-1}}\int_{\Omega} \frac{E(\bx, \by)}{|\bx - \by|^{d-1}}\sum_{k=0}^N \sigma_s(\by) w_k(\by)  V\left(\frac{|\bx-\by|}{h} + k-l\right) d\by \\
&\quad + \frac{1}{\nu_{d-1}}\int_{\Omega} \frac{E(\bx, \by)}{|\bx - \by|^{d-1}}\sum_{k=0}^N c_k(\by)  V\left(\frac{|\bx-\by|}{h} +k-l\right) d\by\,.
\end{aligned}
\end{equation}
It is worth while to notice that the above formulation~\eqref{EQ: WL} implies causality, this is because 
\begin{equation}
 V\left(\frac{|\bx-\by|}{h} +k-l\right) \neq 0 \iff -1 < \frac{|\bx-\by|}{h} +k-l < 1\,,
\end{equation}
hence $l > k + \frac{|\bx - \by|}{h} - 1$, since $l,k\in\bbN$, we must have $l\ge k$. In fact, there are at most two choices for $k$ to take nonzero values, the summation in~\eqref{EQ: WL} over $k$ can be reduced to $k\le l$ instead. The error from discretization in time is estimated in the following Lemma~\ref{LEM: ERROR TIME DISCRET}.
\begin{lemma}\label{LEM: ERROR TIME DISCRET}
    Suppose the assumptions $\cA$ and $\cB$ are satisfied, then 
    \begin{equation}
    \|\aver{u}_h(t,\bx) - \aver{u}(t,\bx)\|_{L^{\infty}([0,T]\times\Omega)} \le \cO(\omega(h))\,.
    \end{equation}
\end{lemma}
\begin{proof}
    With the condition provided for $f$, we can extend $f$ to $(-\infty, T]\times\Omega$ by zero extension without changing the continuity class,  then $\|f_h-f\|_{L^{\infty}((-\infty,T]\times\Omega)} = \cO(\omega(h))$. Take $e_h(t,\bx) = \aver{u}_h(t,\bx) - \aver{u}(t,\bx)$, then
    \begin{equation}
    \begin{aligned}
    e_h(t,\bx) &=  \aver{u}_h(t,\bx) - \aver{u}(t, \bx) \\
    &= \frac{1}{\nu_{d-1}}\int_{\Omega}\frac{E(\bx, \by)}{|\bx - \by|^{d-1}}\sigma_s(\by)  e_h(t_l-{|\bx-\by|},\by )d\by \\
    &\quad + \frac{1}{\nu_{d-1}}\int_{\Omega}\frac{E(\bx, \by)}{|\bx - \by|^{d-1}}\left(  f_h(t_l-{|\bx-\by|},\by) - f(t_l-{|\bx-\by|},\by)\right) d\by\,.
    \end{aligned}
    \end{equation}
    By Lemma~\ref{LEM: WELL}, we obtain that
    \begin{equation}
    \|e_h\|_{L^{\infty}([0,T]\times\Omega)} \le C\|f_h-f\|_{L^{\infty}((-\infty,T]\times\Omega)} = \cO( \omega(h)\,.
    \end{equation}
\end{proof}

\subsection{Discretization of space}
\label{SEC:SPACE}
Clearly, in order to solve the equation~\eqref{EQ: WL}, one has to evaluate the volume integrals on the right-hand-side. We follow the piecewise constant collocation method (PCCM) introduced in~\cite{ren2019fast,vainikko2006multidimensional} for the spatial discretization. The discretization is constructed as follows:
\begin{enumerate}
    \item Partition of $\Omega$. For a small $\ell > 0$, we partition the spatial domain $\Omega$ into two parts: boundary part $\Omega_b^{\ell}$ and interior part $\Omega_{i}^{\ell}$, where
    \begin{equation}
    \Omega_{b}^{\ell} := \{\bx\in\Omega : \text{dist}(\partial\Omega, \bx) \le \ell^2 \}\;\text{ and }\; \Omega_i^{\ell} = \Omega \backslash \Omega_b^{\ell}\,,
    \end{equation}
    Let $\{T_{p,\ell}\}_{p=1}^M$ of $\Omega$ be a spatial discretization, that is $T_{p,\ell}\cap T_{p',\ell}=\emptyset$, $\forall p\neq p'$ and $\Omega = \bigcup_{p=1}^{M} T_{p,\ell}$,  which also satisfies that: (a) $\text{diam}(T_{p,\ell}) \le \ell$, $\forall p$; and (b) $T_{p,\ell}\cap \Omega_i^{\ell}\neq \emptyset$, $\forall p$. (It means no cell $T_{p,\ell}$ is completely in $\Omega_b^{\ell}$, which can be easily satisfied since the thickness of $\Omega_b^{\ell}$ is of order ${\ell}^2$.) It is then clear that $M \simeq \cO({\ell}^{-d})$. For any $1\le p\le M$, if $T_{p,\ell}\cap \Omega_b^{\ell}\neq \emptyset$, we set $T'_{p,\ell} := T_{p,\ell} \cap \Omega_i^{\ell}$ when it is not empty.
    \item Collocation Points. For each cell $T_{p,\ell}$ in the discretization, we locate the collocation point $\bx_p \in T_{p,\ell}$ as follows:
    \begin{enumerate}
        \item If $T_{p,\ell}\subset\Omega_i^{\ell}$,  $\bx_p$ is chosen as the mass centroid point
        \begin{equation}
        \bx_p = \frac{1}{|T_{p,\ell}|} \int_{T_{p,\ell}} \bz d\bz.
        \end{equation}
        \item If $T_{p,\ell}\cap \Omega_{b}^{\ell}\neq \emptyset$, choose an arbitrary $\bx_p\in T'_{p,\ell}$ .
    \end{enumerate}
\end{enumerate}
The simplest example of the above discretization is to use a uniform grid $\cG$ with cell size of $\ell$. For a cell $T_{p,\ell}\subset \cG$ contained in $\Omega$, we choose its centroid point as the collocation point. For a boundary-incident cell $T_{p,\ell}\subset \cG$ such that $T_{p,\ell}\cap \partial\Omega\neq \emptyset$, we replace the cell $T_{p,\ell}$ by the intersection $T_{p,\ell}' = T_{p,\ell}\cap\Omega$ and choose an arbitrary point in $T'_{p,\ell}$ as the collocation point. When the boundary $\partial\Omega$ is $C^2$, the boundary part $\partial\Omega\cap T_{p,\ell}$ can be approximated using a tangent plane or secant plane. The difference of measure in this case is of order $\cO({\ell}^2)$.

\subsection{Linear system from discretization}
\label{SEC:LINEAR SYS}
With the above discretization scheme in space, for each  $l$ that $0\le l\le N$, we approximate $w_{l}(\bx)$ and $c_{l}(\bx)$ with spatially piecewise constant functions $\bar{w}_{l}(\bx)$ and $\bar{c}_{l}(\bx)$,  respectively, 
\begin{equation}\label{EQ:PIECEWISE CONSTANT}
\bar{w}_{l}(\bx) = \sum_{p=1}^{M} {w}_{l}^p \chi_{p}(\bx), \quad  \bar{c}_{l}(\bx) = \sum_{p=1}^{M} {c}_{l}^p \chi_{p}(\bx),\quad \chi_{p}(\bx) = \begin{cases}
1,\quad \bx \in T_{p,\ell} \\
0,\quad \bx \notin T_{p,\ell}
\end{cases}
\end{equation}
Replacing $w_{l}$ by $\bar{w}_{l}$ in the integral equation~\eqref{EQ: WL} and using causality, we obtain the discretized linear equation for $\bar{w}_{l}(\bx_p) = {w}_{l}^p$:
\begin{equation}\label{EQ: DISCRETIZED}
\begin{aligned}
& {w}_{l}^p =\frac{1}{\nu_{d-1}}\sum_{q=1}^M   \sum_{k= 0}^{l}\left( \int_{T_{q,{\ell}}}\frac{E(\bx_p, \by)\sigma_s(\by)  }{|\bx_p - \by|^{d-1}} V_{}\left(\frac{|\bx_p -\by| }{h}+(k-l)\right)  d\by\right){w}_{k}^q \\
&\quad+ \frac{1}{\nu_{d-1}}\sum_{q=1}^M  \sum_{k= 0}^{l}  \left( \int_{T_{q,{\ell}}}\frac{E(\bx_p, \by)  }{|\bx_p - \by|^{d-1}} V_{}\left(\frac{|\bx_p - \by|}{h}+(k-l)\right)  d\by\right) {c}_{k}^q\,.
\end{aligned}
\end{equation}
Similar to~\cite{ren2019fast}, we take the following approximations for the local integrals on $T_{q,\ell}$,
\begin{equation}\nonumber
\begin{aligned}
&\int_{T_{q,{\ell}}}\frac{E(\bx_p, \by)\sigma_s(\by)  }{|\bx_p - \by|^{d-1}}  V_{}\left(\frac{|\bx_p -\by| }{h}+(k-l)\right)  d\by\approx W(\bx_p, \bx_q)\sigma_s(\bx_q)  V_{}\left(\frac{|\bx_p -\bx_q| }{h}+(k-l)\right)  \,,\\
&\int_{T_{q,{\ell}}}\frac{E(\bx_p, \by)  }{|\bx_p - \by|^{d-1}}  V_{}\left(\frac{|\bx_p -\by| }{h}+(k-l)\right)  d\by \approx W(\bx_p, \bx_q)  V_{}\left(\frac{|\bx_p -\bx_q| }{h}+(k-l)\right)  \,,
\end{aligned}
\end{equation}
where $W(\bx_p, \bx_q)$ denotes the following local weakly singular integral,
\begin{equation}
W(\bx_p, \bx_q) = E(\bx_p,\bx_q)\int_{T_{q,\ell}} \frac{1}{|\bx_p - \bz|^{d-1}} d\bz\,.
\end{equation}
Therefore we obtain an explicit linear system of~\eqref{EQ: DISCRETIZED} as follows:
\begin{equation}\label{EQ: DISCRETIZED2}
\begin{aligned}
&{w}_{l}^p =\frac{1}{\nu_{d-1}}\sum_{q=1}^M \sum_{k=0}^l W(\bx_p, \bx_q)   V_{}\left(\frac{|\bx_p -\bx_q| }{h}+(k-l)\right)  (\sigma_s(\bx_q){w}_{k}^q + {c}_{k}^q)\,,
\end{aligned}
\end{equation}
where the left-hand-side is only relevant to the time-step $l$, while the right-hand-side involves the terms on time-steps $k\le l$. Note that the evaluations of $V$ actually only has at most two nonzero values for each pair of $q$ and $l$. In particular, if the time step $h$ is smaller than the grid size $\ell$, then \eqref{EQ: DISCRETIZED2} is an explicit scheme in time.

\begin{remark}
    As~\cite{ren2019fast} points out, the function $W(\bx,\by)$ could be analytically evaluated under certain circumstances. For example, consider the two-dimensional case ($d=2$), and let the discretization $T_{p,\ell}$ be identical and square. Let $T(\by)$ be a square centered at $\by= (y_1, y_2)$ with side length of $\ell$. Let $\bx = (x_1, x_2)$, $t_1 = y_1 - x_1$, and $t_2 = y_2 - x_2$. It is then easy to verify that
    \begin{equation}
    \int_{T(\by)} |\bx - \bz|^{1-d} d\bz =  \sum_{i=-1}^1\sum_{j=-1}^1 ij F\left(t_1 + i\frac{\ell}{2} ,\; t_2 + j\frac{\ell}{2}\right) 
    \end{equation} 
    with the function $F(r, s)$ given by
    \begin{equation}\nonumber
    F(r, s) = \sgn(r) \sgn(s) \Big( |r|\log(|s| + \sqrt{r^2 + s^2})  + |s|\log(|r| +\sqrt{r^2 + s^2})-|r|\log |r| -|s|\log |s| \Big).
    \end{equation} 
    This calculation works for any $(\bx, \by)$ pair over $\Omega\times\Omega$ and any side length $\ell > 0$.
\end{remark}

\subsection{Error analysis}
We now estimate the numerical error for the discretization in~\eqref{EQ: DISCRETIZED2}.  First, we prove the following lemma.
\begin{lemma}\label{LEM: DISCRET}
    For any $k \in (k_0, 1)$, there exists an $\ell_0 > 0$ (depending on $k$) such that for any cell size $\ell \in(0, \ell_0)$ and any $\bx\in\Omega$, 
    \begin{equation}\label{EQ: ELL}
    \frac{1}{\nu_{d-1}}\sum_{q=1}^M \int_{T_{q,\ell}} \frac{E(\bx,\bx_q)}{|\bx-\by|^{d-1}} \sigma_s(\bx_q) d\by < k\,.
    \end{equation}
\end{lemma}
\begin{proof}
    Lemma~\ref{LEM: WELL} says 
    \begin{equation}
    \frac{1}{\nu_{d-1}}\sum_{q=1}^M \int_{T_{q,\ell}} \frac{E(\bx,\by)}{|\bx-\by|^{d-1}} \sigma_s(\by) d\by = \frac{1}{\nu_{d-1}}\int_{\Omega}  \frac{E(\bx,\by)}{|\bx-\by|^{d-1}} \sigma_s(\by) d\by \le k_0\,.
    \end{equation}
    Then
    \begin{equation}
    \begin{aligned}
    \frac{1}{\nu_{d-1}}\sum_{q=1}^M \int_{T_{q,\ell}} \frac{E(\bx,\by)}{|\bx-\by|^{d-1}} \sigma_s(\by) d\by -   \frac{1}{\nu_{d-1}}\sum_{q=1}^M \int_{T_{q,\ell}} \frac{E(\bx,\bx_q)}{|\bx-\by|^{d-1}} \sigma_s(\bx_q) d\by =  \cI_1 + \cI_2\,,
    \end{aligned}
    \end{equation}
    where 
    \begin{equation}
    \begin{aligned}
    \cI_1 &= \frac{1}{\nu_{d-1}}\sum_{q=1}^M \int_{T_{q,\ell}} \frac{E(\bx,\by)-E(\bx,\bx_q)}{|\bx-\by|^{d-1}} \sigma_s(\by) d\by\,,\\
    \cI_2 &= \frac{1}{\nu_{d-1}}\sum_{q=1}^M \int_{T_{q,\ell}} \frac{E(\bx,\bx_q)}{|\bx-\by|^{d-1}} (\sigma_s(\by)-\sigma_{s}(\bx_q) d\by\,,
    \end{aligned}
    \end{equation}
Since $\sigma_t\in C^{1,\alpha}(\overline{\Omega})$, we obtain $\forall \by\in\Omega$, 
\begin{equation}
\sigma_t(\by) - \sigma_t(\bx_q) = \nabla \sigma_t(\bx_q)\cdot (\by - \bx_q) + \cO(|\by-\bx_q|^{1+\alpha})\,.
\end{equation}
then when $\bx_p\neq \bx_q$, $E(\bx_p, \by)$ is differentiable in $T_{q,\ell}$ that
\begin{equation}
\begin{aligned}
E(\bx_p, \by) - E(\bx_p, \bx_q) = \nabla_{\by} E(\bx_p, \bx_q)\cdot (\by - \bx_q) + \cO(|\by - \bx_p|^{1+\alpha})\,,
\end{aligned}
\end{equation}
when $\bx_p = \bx_q$, we derive that
\begin{equation}
E(\bx_p, \by) - E(\bx_p, \bx_q) = E(\bx_p, \by) - 1 = \cO(|\bx_p - \by|)\,.
\end{equation}
Fix $\bx_p$, let the set $N_n := \{\bx_q:(n-1)\ell \le |\bx_p-\bx_q| < n\ell  \}$ which stands for the collocation points within a thin shell, then $|N_n|\le \cO(n^{d-1})$. Then we can estimate $\cI_1$ by
\begin{equation}
\begin{aligned}
|\cI_1| &\le C \sum_{n=1}^{\cO(\ell^{-1})}\sum_{\bx_q\in N_n} \left|\int_{T_{q,\ell}}  \frac{E(\bx_p,\by)-E(\bx_p, \bx_q)}{|\bx_p - \by|^{d-1}} d\by \right| \\&\le C \left|\int_{T_{p,\ell}}  \frac{E(\bx_p,\by)-1}{|\bx_p - \by|^{d-1}} d\by \right| + C \sum_{n=2}^{\cO(\ell^{-1})} (n)^{d-1} \frac{\ell^{1+\alpha}}{ ((n-1)\ell)^{d-1}} \ell^d \\&= \cO(\ell^2) + \cO(\ell^{1+\alpha}) = \cO(\ell^{1+\alpha}).
\end{aligned}
\end{equation}
Use the similar approach, we also obtain $|\cI_2|\le \cO(\ell^{1+\alpha})$. Therefore 
\begin{equation}\label{EQ: DISCRET ERROR}
 \frac{1}{\nu_{d-1}}\sum_{q=1}^M \int_{T_{q,\ell}} \frac{E(\bx,\bx_q)}{|\bx-\by|^{d-1}} \sigma_s(\bx_q) d\by \le k_0 + \cO(\ell^{1+\alpha})\,,
\end{equation}
therefore when $\ell < \cO((k-k_0)^{1/(1+\alpha)})$ is sufficiently small, the right hand side of~\eqref{EQ: DISCRET ERROR} is strictly less than $k$.
\end{proof}
\begin{theorem}\label{THM: ERR}
    Suppose $\ell$ is small enough such that~\eqref{EQ: ELL} is satisfied. Define the spatially piecewise constant solution $\bar{w}_h(t, \bx)$
    \begin{equation}
    \bar{w}_h(t,\bx) = \sum_{l=0}^N \sum_{p=1}^M \bar{w}^p_{l}\chi_p(\bx) \phi_{l}(t)\,,
    \end{equation}
    where $\bar{w}_{l}^p$ is defined as~\eqref{EQ:PIECEWISE CONSTANT}.
    Then $|\bar{w}_h(t, \bx_p) - \aver{u}_h(t,\bx_p)|\le \cO(\omega(h) +\omega(\ell))$ for any collocation point $\bx_p$.
\end{theorem}
\begin{proof}
Define the difference $\bar{e}(t,\bx) := \aver{u}_h(t,\bx) - \bar{w}_h(t,\bx)$, which satisfies
\begin{equation}\nonumber
\begin{aligned}
\bar{e}(t,\bx_p) &= \frac{1}{\nu_{d-1}}\sum_{q=1}^M \int_{T_{q,\ell}} \frac{E(\bx_p, \by)}{|\bx_p - \by|^{d-1}}\left( \sigma_s(\by) \aver{u}_h(t-{|\bx_p - \by|}, \by) + f_h(t-{|\bx_p - \by|}, \by) \right) d\by \\
& \quad - \frac{1}{\nu_{d-1}}\sum_{q=1}^M \int_{T_{q,\ell}} \frac{E(\bx_p, \bx_q)}{|\bx_p - \by|^{d-1}}\left( \sigma_s(\bx_q) \bar{w}_h(t-{|\bx_p - \bx_q|}, \bx_q) + f_h(t-{|\bx_p - \bx_q|}, \bx_q) \right) d\by \\
&= \cK_{h,\ell} \overline{e} +\cI_1 +\cI_2 +\cI_3\,,
\end{aligned}
\end{equation}
where the integral operator $\cK_{h,\ell}$ is
\begin{equation}
\cK_{h.\ell}\overline{e} := \frac{1}{\nu_{d-1}}\sum_{q=1}^M \int_{T_{q,\ell}} \frac{E(\bx_p, \bx_q)}{|\bx_p - \by|^{d-1}}\left( \sigma_s(\bx_q) \bar{e}(t-{|\bx_p - \bx_q|},\bx_q) \right)d\by\,,
\end{equation}
which is a contraction operator in $L^{\infty}([0,T]\times\Omega)$ by Lemma~\ref{LEM: DISCRET}. The quantities $\cI_i, i=1,2,3$ are defined as:
\begin{equation}\nonumber
\begin{aligned}
\cI_1 &= \frac{1}{\nu_{d-1}}\sum_{q=1}^M \int_{T_{q,\ell}} \frac{E(\bx_p,\by)-E(\bx_p, \bx_q)}{|\bx_p - \by|^{d-1}}\left( \sigma_s(\by) \aver{u}_h(t-{|\bx_p - \by|},\by) + f_h(t-{|\bx_p - \by|}, \by) \right)d\by\,, \\
\cI_2 &= \frac{1}{\nu_{d-1}}\sum_{q=1}^M \int_{T_{q,\ell}}\frac{E(\bx_p,\bx_q)}{|\bx_p-\by|^{d-1}}\left(f_h(t-{|\bx_p-\by|},\by) -  f_h(t-{|\bx_p-\bx_q|},\bx_q)\right)d\by\,,\\
\cI_3 &= \frac{1}{\nu_{d-1}}\sum_{q=1}^M \int_{T_{q,\ell}}\frac{E(\bx_p,\bx_q)}{|\bx_p-\by|^{d-1}}\left(\sigma_s(\by)\aver{u}_h(t-{|\bx_p-\by|},\by) -  \sigma_s(\bx_q)\aver{u}_h(t-{|\bx_p-\bx_q|},\bx_q)\right)d\by\,.
\end{aligned}
\end{equation}
The estimate $|\cI_1|\le \cO(\ell^{1+\alpha})$ is the same as the in the proof of Lemma~\ref{LEM: DISCRET}.
For the estimate of  $\cI_2$, it is simple to derive that
\begin{equation}
|f_h(t-{|\bx_p-\by|},\by) -  f_h(t-{|\bx_p-\bx_q|},\bx_q)| \le \cO(\omega(h)) + \cO(\omega(\ell))
\end{equation}
using the relation between $f$ and $f_h$. Therefore $|\cI_2|\le \cO(\omega(h) + \omega(\ell))$. By Lemma~\ref{LEM:TIME} and Lemma~\ref{LEM: SPACE}, we have $|\cI_3| \le \cO(\omega(h)+\omega(\ell) )$ as well. 
From the contraction property of $\cK_{h,\ell}$, we get the estimate of the numerical error of the discretization schemes $\bar{e}(t,\bx)\le \cO(\omega(h) + \omega(\ell))$.
\end{proof}
\subsection{Treecode algorithm}


To solve the linear system~\eqref{EQ: DISCRETIZED2} at the $l$-th time step, we need to evaluate the summation on the right-hand-side with $k\le l$ which is the main cost when using the integral formulation. Direct evaluation of such a summation will take at least $\cO(M^2)$ operations. We will accelerate the summation in~\eqref{EQ: DISCRETIZED2} using the treecode algorithm~\cite{barnes1986hierarchical}. 

In the treecode algorithm, the point set $\{ \bx_p \}_{p=1}^M$ is partitioned into hierarchical clusters $\{C_s\}$ with $k$-d tree structure, we call the leaves of the $k$-d tree as leaf-clusters, if a cluster $C_t$ is produced by partitioning $C_s$ directly, then we call $C_t$ as a child-cluster of $C_s$. 
To determine whether or not a cluster $C_s$ is in the far-field of the point $\bx_p$, we let $r$ be the radius of cluster $C_s$ and $R$ is the distance between the center of $C_s$ to $\bx_p$, when $r/R\le \theta$ for some user-specified parameter $\theta$, then the cluster $C_s$ is assumed to be in the far field of $\bx_p$, otherwise the cluster is assumed to be in the near field of $\bx_p$. Using the hierarchical structure, the summation in the following form
\begin{equation}\label{EQ:SUMMATION}
u(\bx_p) = \sum_{q=1}^M K(\bx_p, \bx_q) f( \bx_q), \quad 1\le p\le M
\end{equation}
can be re-grouped into
\begin{equation}
u(\bx_p) = \sum_{s\in S} {u}_s(\bx_p),\quad {u}_s(\bx_p):=\sum_{\bx_q\in C_s}K(\bx_p, \bx_q) f(\bx_q),
\end{equation}
where $S$ is a certain index set depending on $\bx_p$, $C_s$ denotes a cluster in the far-field of $\bx_p$ or a leaf-cluster such that $\Large\cup_{s\in S} C_s = \{ \bx_p \}_{p=1}^M$ and $C_s \Large\cap C_t = \emptyset$ when $s\neq t$. When the cluster $C_s$ is in the far-field of the point $\bx_p$, the quantity $u_s$ can be evaluated through some approximation with a lower computational cost. Here we follow the interpolative idea in Fong-Darve's fast multipole method~\cite{fong2009black} and~\cite{wang2019kernel}. Let $T_k$ be the first-kind Chebyshev polynomial of degree $k$ defined on $[-1,1]$, then we define the interpolation function 
\begin{equation}\label{eq:Cheby}
S_n(\bx, \by) = \prod_{i=1}^{d}\left(\frac{1}{n} +\frac{2}{n}\sum_{k=1}^{n-1}T_{k}(x_i) T_k(y_i)\right)\,,
\end{equation}
where $\bx = (x_1,\dots, x_d)\in [-1,1]^d$, $\by = (y_1,\dots, y_d)\in[-1,1]^d$. Assume the cluster $C_j$ is contained in a hypercube $X_j$ that $C_j\subset X_j:=  \prod_{i=1}^{d}[a_i, b_i]$ which stays in the far-filed of $\bx_p$, we can define the linear transformation $\cL_j: [-1,1]^d\to X$ that 
\begin{equation}
\cL_j \bx = \frac{\ba + \bb}{2} + \frac{\bb - \ba}{2}\odot \bx, \quad \ba = (a_1,\dots, a_d),\quad \bb = (b_1,\dots, b_d),
\end{equation} 
where $\odot$ denote Hadamard product between two vectors. The transformation $\cL$ maps the standard Chebyshev points in $[-1,1]^d$ to the scaled Chebyshev points in $X$. Then we can approximate $u_s$ by the following interpolation formulation~\cite{dutt1996fast},
\begin{equation}\label{eq:int}
u_s\approx \sum_{m=1}^{n^d} K(\bx_p, \cL_s\by_m) Z_{s,m}
\end{equation}
with $Z_{s,m}$ evaluated in the following two cases:
\begin{enumerate}
    \item If $C_s \text{ is a leaf-cluster}$, then 
    $$ Z_{s,m} = 
    \sum_{\bx_q\in C_s} S_n(\cL_s^{-1}\bx_q, \by_m)  {f}(\bx_q)\,,$$
    \item Otherwise, 
    $$
    Z_{s,m} = \sum_{   \substack{t \\ C_t \text{ is a child-cluster of } C_s} }\sum_{r=1}^{n^d} S_n(\cL_s^{-1} L_t \by_r, \by_m) Z_{t, r}\,,
    $$
\end{enumerate}
where $\{ \by_m \}_{m=1}^{n^d}\in [-1,1]^d$ is the set of $d$-dimensional Chebyshev interpolation points formed by tensor product of the $n$th order Chebyshev points on $[-1,1]$. Then the summation~\eqref{EQ:SUMMATION} can be approximated by
\begin{equation}
u(\bx_p)\approx \sum_{s\in S} \sum_{m=1}^{n^d} K(\bx_p, \cL_s\by_m) Z_{s,m}\,.
\end{equation}
Applying this approximation to~\eqref{EQ: DISCRETIZED2}, we obtain the accelerated summation for the right-hand-side. Since the local summations $Z_{s,m}$ can be precomputed from bottom to top on the $k$-d tree with $\cO(M)$ time complexity, the computation complexity will be reduced to $\cO( M\log M)$ for each time step. If the time step $h\simeq \cO(\ell)$, then total number of time interval $N\simeq \cO(M^{1/d})$, which means the total computation complexity with treecode algorithm is $\cO(M^{1+1/d}\log M)$.

\begin{remark}
    In the above treecode algorithm, the evaluation of $E(\bx, \by)$ in the kernel function $K(\bx,\by)$ is of computation cost $\cO(1)$ if the integral of $\sigma_t$ is known analytically or the involved evaluation is precomputed. In practice, if the coefficient $\sigma_t$ is only known on the collocation points, the evaluation cost of $E(\bx, \by)$ is proportional to the number of grids  along the segment connecting $\bx$ and $\by$, which is $\cO(\frac{|\bx - \by|}{\ell})$ with a naive summation. In this case, the total precomputing cost is at $\cO(M^{1+1/d}\log M)$. Hence the total computational complexity is still $\cO(M^{1+1/d}\log M)$ if $h\simeq \cO(\ell)$.
\end{remark}

\begin{remark}\label{re:kernel}
The accuracy of using the Chebyshev polynomial~\eqref{eq:Cheby} in the interpolation formulation~\eqref{eq:int} depends on the smoothness of the kernel function. Since piecewise linear interpolation is used in time, the kernel function is only piecewise linear in space. Hence,  the numerical solution to the treecode algorithm will not approximate the true solution of the discretized linear system very accurately even if high order interpolation is used. However, as long as this interpolation error matches the numerical discretization error for the integral equation, for example, see the Experiment II of Sec~\ref{sec:test}, when the interpolation order $n=3$, our treecode based algorithm already provides a fast solver for time-dependent RTE. When the mesh becomes finer and finer, the order of interpolation may need to be increased to maintain the accuracy.

\end{remark}

\section{Numerical experiments}
\label{sec:test}
In this section, we demonstrate the fast algorithm for the time-dependent radiative transport equation with numerical experiments in 2D\footnote{The code repository is hosted at \href{https://github.com/lowrank/treecode_rte}{https://github.com/lowrank/treecode\_rte}.}, the numerical experiments are implemented in C++, the treecode algorithm is naturally parallelized with OpenMP.  The computational domain is fixed as $\Omega = [0,1]^2$ and $T = 1$ for the following experiments. For simplicity, we take the uniform spatial discretization with cell length $\ell=M^{-1/2}$ and the time step $h = \ell$, where $M$ denotes the total number of collocation points in $\Omega$. We also denote $t_{dir}$ as the running time by computing~\eqref{EQ: DISCRETIZED2} directly and $t_{tree}$ as the running time for~\eqref{EQ: DISCRETIZED2} with the treecode algorithm. The source function $f$ is chosen as the following:
\begin{equation}
   f(t,\bx) = 4 t^2 \exp(-40|\bx - \br(t)|^2)\,,
\end{equation}
where $\br(t) = (\frac{1}{2} + \frac{1}{5}\cos(4\pi t), \frac{1}{2} + \frac{1}{5}\sin(4\pi t))$, which represents a Gaussian point source with increasing intensity traveling two rounds around the center of $\Omega$. All the numerical experiments are performed using a desktop with $12$ Intel Xeon CPUs at $2.27$GHz  and $32$GB memory.
\subsection{Experiment I}
In the first experiment, we take $\sigma_s \equiv 5.0$ and $\sigma_t\equiv 5.2$. Let $\aver{u}_{dir}$ and $\aver{u}_{tree}$ be the discretized solutions to~\eqref{EQ: DISCRETIZED2} with and without treecode algorithm to accelerate respectively. We use
\begin{equation}
    E_{\ell^{2}} = \frac{\|\aver{u}_{dir} - \aver{u}_{tree}\|_{\ell^{2}} }{\|\aver{u}_{dir}\|_{\ell^2}}
\end{equation}
to measure the difference.  We show comparisons of solutions in two cases. 

\noindent {\bf{Case I.\;}}  We fix the Chebyshev polynomial interpolation of order $n=6$ and let the parameter $\theta$ (the ratio of the cluster size and the separation distance) in treecode algorithm take values: $\theta=0.3,0.4,0.5,0.6,0.7$. The numerical results are shown in Tab~\ref{TAB:1}. Since the coefficient $\sigma_t$ is a constant, we evaluate $E(\bx,\by) = \exp(-\sigma_t|\bx-\by|)$ directly, one can observe that the growth of running time $t_{dir}$ with respect to  $M$ is almost at order of $\cO(M^{5/2})$ and the growth of $t_{tree}$ is relatively slower. As the parameter $\theta$ decreases, the approximation error $E_{\ell^{2}}$ becomes smaller. 
\begin{table}[!htb]
    \centering
    \caption{The computational time and relative error between the solutions with and without the treecode algorithm under different values of the parameter $\theta$ and collocation points $M$. The Chebyshev polynomial interpolation's order is fixed as $n=6$. }\vspace{0.1cm}
    \label{TAB:1}
    \begin{tabular}{c c c c c c c}
        \hline
        $M$ & $n$ & $\theta$ & $t_{dir}(s)$ & $t_{tree}(s)$ & $E_{\ell^{2}}$ \rule{0pt}{2.6ex}\rule[-1.2ex]{0pt}{0pt}\\
        \hline
         \rule{0pt}{4ex}$2,304$ & $6$ & $0.7$ & $24.30$ & $3.96$ & $4.59\times 10^{-2}$ \\
        $2,304$ & $6$ & $0.6$ & $24.30$ & $4.70$ & $2.59\times 10^{-2}$ \\
        $2,304$ & $6$ & $0.5$ & $24.30$ & $5.80$ & $1.20\times 10^{-2}$\\
        $2,304$ & $6$ & $0.4$ & $24.30$ & $7.25$ & $3.88\times 10^{-3}$\\
        $2,304$ & $6$ & $0.3$ & $24.30$ & $9.36$ & $6.53\times 10^{-4}$\\
        \vspace{0.2cm}\\
        $4,096$ & $6$ & $0.7$ & $95.15$ & $10.85$ & $6.07\times 10^{-2}$\\
        $4,096$ & $6$ & $0.6$ & $95.15$ & $12.89$ & $3.99\times 10^{-2}$\\
        $4,096$ & $6$ & $0.5$ & $95.15$ & $16.16$ & $2.43\times 10^{-2}$\\
        $4,096$ & $6$ & $0.4$ & $95.15$ & $24.30$ & $1.28\times 10^{-2}$\\
        $4,096$ & $6$ & $0.3$ & $95.15$ & $29.05$ & $4.81\times 10^{-3}$\\
        \vspace{0.2cm}\\
        $6,400$ & $6$ & $0.7$ & $300.7$ & $23.94$ & $6.27\times 10^{-2}$\\
        $6,400$ & $6$ & $0.6$ & $300.7$ & $29.84$ & $4.16\times 10^{-2}$\\
        $6,400$ & $6$ & $0.5$ & $300.7$ & $38.43$ & $2.54\times 10^{-2}$\\
        $6,400$ & $6$ & $0.4$ & $300.7$ & $50.62$ & $1.35\times 10^{-2}$\\
        $6,400$ & $6$ & $0.3$ & $300.7$ & $70.25$ & $5.01\times 10^{-3}$\\
        \hline
    \end{tabular}
\end{table}

\vspace{0.2cm}
\noindent{\bf{Case II.\;}} We fix the parameter $\theta = 0.3$ and let the Chebyshev polynomial's order $n$ take values $n=2,3,4,5,6$. Similar to Case I, since the coefficient $\sigma_t$ is a constant, we have evaluated $E(\bx,\by) = \exp(-\sigma_t|\bx-\by|)$ directly in the experiment. The numerical results are shown in Tab~\ref{TAB:2}. We can find that increasing the order $n$ of Chebyshev polynomial is not effectively reducing the approximation error, this is due to the deficiency of smoothness of the integral kernel in~\eqref{EQ: DISCRETIZED2} as explained in Remark~\ref{re:kernel}. 
\begin{table}[!htb]
    \centering
    \caption{The computational time and relative error between the solutions with and without the treecode algorithm under different interpolation orders $n$ and collocation points $M$. The treecode algorithm related parameter $\theta$ is fixed as $\theta=0.3$. }\vspace{0.1cm}
    \label{TAB:2}
    \begin{tabular}{c c c c c c c}
        \hline
        $M$ & $n$ & $\theta$ & $t_{dir}(s)$ & $t_{tree}(s)$ & $E_{\ell^{2}}$ \rule{0pt}{2.6ex}\rule[-1.2ex]{0pt}{0pt}\\
        \hline
        \rule{0pt}{4ex}$2,304$ & $2$ & $0.3$ & $24.30$ & $3.08$ & $1.54\times 10^{-2}$ \\
        $2,304$ & $3$ & $0.3$ & $24.30$ & $4.86$ & $4.48\times 10^{-3}$ \\
        $2,304$ & $4$ & $0.3$ & $24.30$ & $6.03$ & $4.47\times 10^{-3}$\\
        $2,304$ & $5$ & $0.3$ & $24.30$ & $7.89$ & $4.48\times 10^{-3}$\\
        $2,304$ & $6$ & $0.3$ & $24.30$ & $9.36$ & $6.53\times 10^{-4}$\\
        \vspace{0.2cm}\\
        $4,096$ & $2$ & $0.3$ & $95.15$ & $8.75$ & $1.18\times 10^{-2}$\\
        $4,096$ & $3$ & $0.3$ & $95.15$ & $14.17$ & $8.08\times 10^{-3}$\\
        $4,096$ & $4$ & $0.3$ & $95.15$ & $20.58$ & $4.80\times 10^{-3}$\\
        $4,096$ & $5$ & $0.3$ & $95.15$ & $24.82$ & $4.81\times 10^{-3}$\\
        $4,096$ & $6$ & $0.3$ & $95.15$ & $29.05$ & $4.81\times 10^{-3}$\\
        \vspace{0.2cm}\\
        $6,400$ & $2$ & $0.3$ & $300.7$ & $21.28$ & $1.32\times 10^{-2}$\\
        $6,400$ & $3$ & $0.3$ & $300.7$ & $36.12$ & $7.22\times 10^{-3}$\\
        $6,400$ & $4$ & $0.3$ & $300.7$ & $50.95$ & $7.21\times 10^{-3}$\\
        $6,400$ & $5$ & $0.3$ & $300.7$ & $61.36$ & $5.01\times 10^{-3}$\\
        $6,400$ & $6$ & $0.3$ & $300.7$ & $70.25$ & $5.01\times 10^{-3}$\\
        \hline
    \end{tabular}
\end{table}

\subsection{Experiment II}
In this experiment, we study the self-convergence tests on the accuracy of the solutions to~\eqref{EQ: DISCRETIZED2} with the treecode algorithm. We perform the numerical simulation with $\sigma_s \equiv 5.0$ and $\sigma_t\equiv 5.2$ for different cell sizes $\ell = \frac{1}{24(2k-1)}$ for $1\le k \le 8$ and different Chebyshev polynomial interpolation order: $n=3$, $n=4$, $n=5$. The solution at $k=8$ is taken as the reference solution, the numerical relative errors are evaluated using the $\ell^{2}$-norm  on the common collocation points at the coarsest level $k=1$. We can observe that the convergence is faster than linear (see Fig~\ref{FIG:1}). This is partly because the collocation points on the coarsest level are reasonably far from the boundary, thus the numerical solutions on these points are less affected by the boundary effect.

\begin{figure}[!htb]
    \centering
    \includegraphics[scale=0.4]{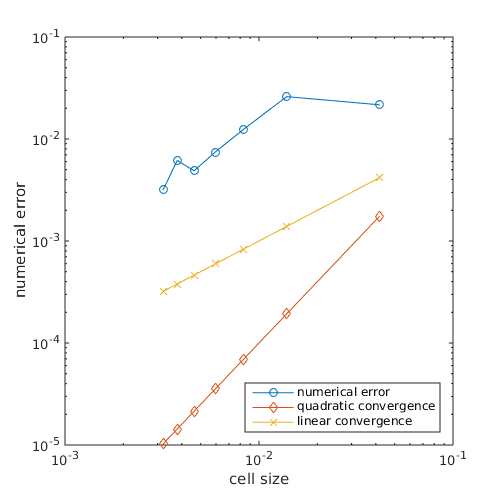}
    \includegraphics[scale=0.4]{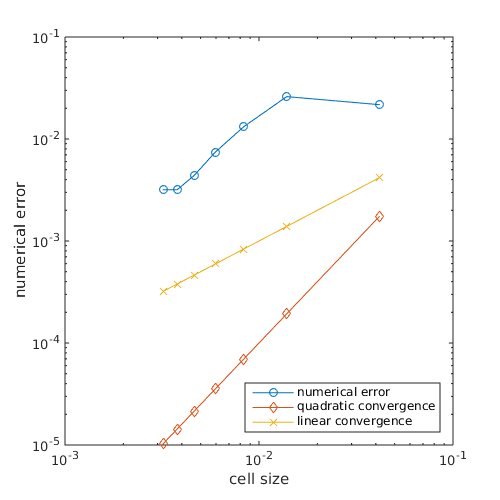}
    \includegraphics[scale=0.4]{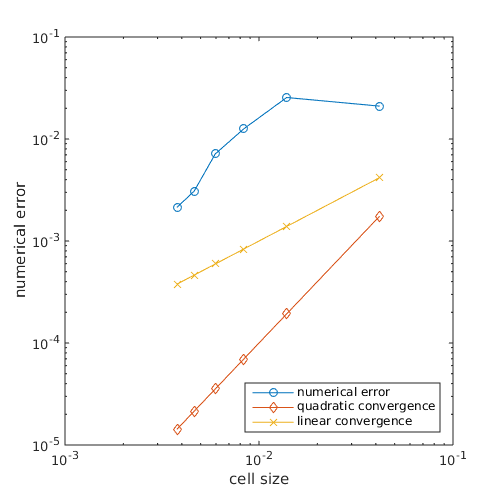}
    \caption{The numerical errors with respect to various grid sizes with $\sigma_t = 5.2$ and $\sigma_s = 5.0$. The relative $\ell^2$ error of the solutions are compared
        with the reference solution calculated at the finest level $k=8$.
        The Chebyshev polynomial interpolation order from left to right are:  $n=3$, $n=4$, $n=5$.}
    \label{FIG:1}
\end{figure}

\section{Concluding remarks}
\label{sec:conclusion}
In this work, we develop a fast algorithm to solve the time-dependent radiative transport equation in isotropic media. The method is based on the integral formulation~\eqref{EQ: INT EQ} and uses the treecode algorithm to accelerate the computation. Numerical experiments are performed to show the efficiency and accuracy of the algorithm. We emphasize that the integral formulation does not rely on the assumption of infinite homogeneous media, which is different from existing methods~\cite{tan2001integral,wu2000integral}. For inhomogeneous media, the treecode algorithm involves evaluations of path integrals for different pairs of $(\bx, \by)$. Although the computation cost is increased compared to that in the homogeneous case, those evaluations can be precomputed once and reused for each time step. The total computational cost is the same order as the case of homogeneous media, which is $\cO(M^{1+1/d}\log M)$, where $M$ is the number of collocation points in the physical space. 

The main contribution of this work is on the combination of the integral formulation~\eqref{EQ: INT EQ} and the treecode algorithm to accelerate the solution for the time-dependent radiative transport equation. We believe there are other ways to solve the time-dependent radiative transport equation efficiently, e.g., solving the equation in frequency domain with the idea from~\cite{candes2009fast}, which will be studied in our future work.

\section*{Acknowledgments}
Hongkai Zhao is partially supported by NSF DMS-1821010.

\bibliographystyle{siam}
\bibliography{main}

\end{document}